\documentclass{amsart}
\usepackage{amssymb}
\usepackage{amsmath}
\usepackage{amsfonts}

\newtheorem{theorem}{Theorem}
\theoremstyle{plain}

\newtheorem{definition}{Definition}
\newtheorem{example}{Example}

\newtheorem{proposition}{Proposition}
\newtheorem{remark}{Remark}

\numberwithin{equation}{section}
\input{tcilatex}

\begin{document}
\title{Cantor--Kuratowski theorem in uniformizable spaces}
\author{Josiney A. Souza}
\author{Richard W. M. Alves}
\address{Universidade Estadual de Maring\'{a}, Brazil. \\
Email: joasouza3@uem.br}

\begin{abstract}
This manuscript extends the Cantor--Kuratowski intersection theorem from the
setting of metric spaces to the setting of uniformizable spaces. Complete
uniformizable spaces are revisited.
\end{abstract}

\maketitle

\section{Introduction}

The notion of uniformity was provided with the intention of transferring
methodologies from metric space to general topological spaces under absence
of metrization. Recently, the uniformizable spaces were described as
admissible spaces, which are topological spaces admitting admissible family
of open coverings (\cite{Richard}). Essentially, a covering uniformity is an
admissible family of open covering, while an admissible family of open
covering provides a base for a covering uniformity. The present paper
employs the admissible structure to extend the classical \emph{%
Cantor--Kuratowski intersection theorem} to the setting of uniformizable
spaces.

Let $\left( M,\mathrm{d}\right) $ be a metric space. The \emph{Cantor
intersection theorem} says that $M$ is a complete metric space if and only
if every decreasing sequence $F_{1}\supset F_{2}\supset \cdots \supset
F_{n}\supset \cdots $ of nonempty closed sets of $M$, with $\mathrm{diam}%
\left( F_{n}\right) \rightarrow 0$, has nonempty intersection. This theorem
was generalized by K. Kuratowski \cite{Kuratowski} by using the concept of
measure of noncompactness. The \emph{Kuratowski measure of noncompactness}
of $Y$ is defined as 
\begin{equation*}
\gamma \left( Y\right) =\inf \left\{ \delta >0:Y\text{ admits a finite cover
by sets of diameter at most }\delta \right\} .
\end{equation*}%
Then the measure of noncompactness associates numbers to sets in such a way
that every compact sets has measure zero while any noncompact set has
measure greater than zero. The \emph{Cantor--Kuratowski intersection theorem}
says that the metric space $M$ is complete if and only if every decreasing
sequence $\left( F_{n}\right) $ of nonempty bounded closed sets, with $%
\gamma \left( F_{n}\right) \rightarrow 0$, has nonempty compact
intersection. In the present paper, we use the admissible structure to
reproduce both the concepts of diameter and measure of noncompactness to
extend the intersection theorems to the general setting of uniformizable
spaces.

\section{\label{Section1}Admissible structure}

This section contains the basic definitions and results on admissible
spaces. We refer to \cite{Richard}, \cite{patrao2}, and \cite{So} for the
previous development of admissible spaces.

Let $X$ be a topological space and $\mathcal{U},\mathcal{V}$ coverings of $X$%
. We write $\mathcal{V}\leqslant \mathcal{U}$ if $\mathcal{V}$ is a
refinement of $\mathcal{U}$. One says $\mathcal{V}$ double-refines $\mathcal{%
U}$, or $\mathcal{V}$ is a \textbf{double-refinement} of $\mathcal{U}$,
written $\mathcal{V}\leqslant \frac{1}{2}\mathcal{U}$ or $2\mathcal{V}%
\leqslant \mathcal{U}$, if for every $V,V^{\prime }\in \mathcal{V}$, with $%
V\cap V^{\prime }\neq \emptyset $, there is $U\in \mathcal{U}$ such that $%
V\cup V^{\prime }\subset U$. We write $\mathcal{V}\leqslant \frac{1}{2^{2}}%
\mathcal{U}$ if there is a covering $\mathcal{W}$ of $X$ such that $\mathcal{%
V}\leqslant \frac{1}{2}\mathcal{W}$ and $\mathcal{W}\leqslant \frac{1}{2}%
\mathcal{U}$. Inductively, we write $\mathcal{V}\leqslant \frac{1}{2^{n}}%
\mathcal{U}$ if there is $\mathcal{W}$ with $\mathcal{V}\leqslant \frac{1}{2}%
\mathcal{W}$ and $\mathcal{W}\leqslant \frac{1}{2^{n-1}}\mathcal{U}$. In
certain sense, the notion of double-refinement in topological spaces
corresponds to the property of triangle inequality in metric spaces.

Now, for a covering $\mathcal{U}$ of $X$ and a subset $Y\subset X$, the 
\emph{star} of $Y$ with respect to $\mathcal{U}$ is the set 
\begin{equation*}
\mathrm{St}\left[ Y,\mathcal{U}\right] =\bigcup \left\{ U\in \mathcal{U}%
:Y\cap U\neq \emptyset \right\} \text{.}
\end{equation*}%
If $Y=\left\{ x\right\} $, we usually write $\mathrm{St}\left[ x,\mathcal{U}%
\right] $ rather than $\mathrm{St}\left[ \left\{ x\right\} ,\mathcal{U}%
\right] $. Then one has $\mathrm{St}\left[ Y,\mathcal{U}\right]
=\bigcup\limits_{x\in Y}\mathrm{St}\left[ x,\mathcal{U}\right] $ for every
subset $Y\subset X$.

\begin{definition}
\label{Admiss} A family $\mathcal{O}$ of open coverings of $X$ is said to be 
\textbf{admissible} if it satisfies the following properties:

\begin{enumerate}
\item For each $\mathcal{U}\in \mathcal{O}$, there is $\mathcal{V}\in 
\mathcal{O}$ such that $\mathcal{V}\leqslant \frac{1}{2}\mathcal{U}$;

\item If $Y\subset X$ is an open set and $K\subset Y$ is a compact subset of 
$X$ then there is an open covering $\mathcal{U}\in \mathcal{O}$ such that $%
\mathrm{St}\left[ K,\mathcal{U}\right] \subset Y$;

\item For any $\mathcal{U},\mathcal{V}\in \mathcal{O}$, there is $\mathcal{W}%
\in \mathcal{O}$ such that $\mathcal{W}\leqslant \mathcal{U}$ and $\mathcal{W%
}\leqslant \mathcal{V}$.
\end{enumerate}

The space $X$ is called \textbf{admissible} if it admits an admissible
family of open coverings.
\end{definition}

The properties 1 and 2 of Definition \ref{Admiss} guarantee that the stars $%
\mathrm{St}\left[ x,\mathcal{U}\right] $, for $x\in X$ and $\mathcal{U}\in 
\mathcal{O}$, form a basis for the topology of $X$, while the property 3
assures that the family $\mathcal{O}$ is directed by refinements.

It is easily seen that the properties 1 and 3 together are equivalent to the
following one:

\begin{enumerate}
\item[$\left( 1\right) ^{\prime }$] For any $\mathcal{U},\mathcal{V}\in 
\mathcal{O}$, there is $\mathcal{W}\in \mathcal{O}$ such that $\mathcal{W}%
\leqslant \frac{1}{2}\mathcal{U}$ and $\mathcal{W}\leqslant \frac{1}{2}%
\mathcal{V}$.
\end{enumerate}

Then Definition \ref{Admiss} can be simplified by requiring that the family $%
\mathcal{O}$ of open coverings of $X$ satisfies the following:

\begin{enumerate}
\item[$\left( 1\right) ^{\prime }$] For any $\mathcal{U},\mathcal{V}\in 
\mathcal{O}$, there is $\mathcal{W}\in \mathcal{O}$ such that $\mathcal{W}%
\leqslant \frac{1}{2}\mathcal{U}$ and $\mathcal{W}\leqslant \frac{1}{2}%
\mathcal{V}$.

\item[$\left( 2\right) ^{\prime }$] The stars $\mathrm{St}\left[ x,\mathcal{U%
}\right] $, for $x\in X$ and $\mathcal{U}\in \mathcal{O}$, form a basis for
the topology of $X$.
\end{enumerate}

\begin{example}
\label{Ex1}

\begin{enumerate}
\item If $X$ is a paracompact Hausdorff space, then the family $\mathcal{O}$
of all open coverings of $X$ is admissible.

\item If $X$ is a compact Hausdorff space, then the family $\mathcal{O}_{f}$
of all finite open coverings of $X$ is admissible.

\item If $\left( X,\mathrm{d}\right) $ is a metric space, then the family $%
\mathcal{O}_{\mathrm{d}}$ of the coverings $\mathcal{U}_{\varepsilon
}=\left\{ \mathrm{B}_{\mathrm{d}}\left( x,\varepsilon \right) :x\in
X\right\} $ by $\varepsilon $-balls, for $\varepsilon >0$, is admissible.
For every $\varepsilon >0$ and $Y\subset X$ we have $\mathcal{U}%
_{\varepsilon /2}\leqslant \frac{1}{2}\mathcal{U}_{\varepsilon }$ and 
\begin{equation*}
\mathrm{St}\left[ Y,\mathcal{U}_{\varepsilon /2}\right] \subset \mathrm{B}_{%
\mathrm{d}}\left( Y,\varepsilon \right) \subset \mathrm{St}\left[ Y,\mathcal{%
U}_{\varepsilon }\right] \text{.}
\end{equation*}

\item If $X$ is a uniformizable space then the covering uniformity of $X$ is
an admissible family of open coverings of $X$.
\end{enumerate}
\end{example}

\begin{remark}
\label{R1} Since the collection $\left\{ \mathrm{St}\left[ x,\mathcal{U}%
\right] :\mathcal{U}\in \mathcal{O}\right\} $ is a neighborhood base at $%
x\in X$, one has $\bigcap\limits_{\mathcal{U}\in \mathcal{O}}\mathrm{St}%
\left[ Y,\mathcal{U}\right] =\mathrm{cls}\left( Y\right) $ for every subset $%
Y\subset X$. If $X$ is Hausdorff, it follows that $\bigcap\limits_{\mathcal{U%
}\in \mathcal{O}}\mathrm{St}\left[ x,\mathcal{U}\right] =\left\{ x\right\} $
for every $x\in X$.
\end{remark}

Let $X$ be a fixed completely regular space endowed with an admissible
family of open coverings $\mathcal{O}$. Let $\mathcal{P}\left( \mathcal{O}%
\right) $ denote the power set of $\mathcal{O}$ and consider the partial
ordering relation on $\mathcal{P}\left( \mathcal{O}\right) $ given by
inverse inclusion: for $\mathcal{E}_{1},\mathcal{E}_{2}\in \mathcal{P}\left( 
\mathcal{O}\right) $%
\begin{equation*}
\mathcal{E}_{1}\prec \mathcal{E}_{2}\text{ if and only if }\mathcal{E}%
_{1}\supset \mathcal{E}_{2}.
\end{equation*}%
Concerning this relation, $\mathcal{O}$ is the smallest element in $\mathcal{%
P}\left( \mathcal{O}\right) $, or in other words, $\mathcal{O}$ is the lower
bound for $\mathcal{P}\left( \mathcal{O}\right) $. On the other hand, the
empty set $\emptyset $ is the upper bound for $\mathcal{P}\left( \mathcal{O}%
\right) $. Intuitively, $\mathcal{O}$ is the \textquotedblleft
zero\textquotedblright\ and $\emptyset $ is the \textquotedblleft
infinity\textquotedblright .

For each $\mathcal{E}\in \mathcal{P}\left( \mathcal{O}\right) $ and $n\in 
\mathbb{N}^{\ast }$ we define the set $n\mathcal{E}$ in $\mathcal{P}\left( 
\mathcal{O}\right) $ by 
\begin{equation*}
n\mathcal{E}=\left\{ \mathcal{U}\in \mathcal{O}:\text{there is }\mathcal{V}%
\in \mathcal{E}\text{ such that }\mathcal{V}\leqslant \tfrac{1}{2^{n}}%
\mathcal{U}\right\} .
\end{equation*}

This operation is order-preserving, that is, if $\mathcal{E}\prec \mathcal{D}
$ then $n\mathcal{E}\prec n\mathcal{D}$. In fact, if $\mathcal{U}\in n%
\mathcal{D}$ then there is $\mathcal{V}\in \mathcal{D}$ such that $\mathcal{V%
}\leqslant \tfrac{1}{2^{n}}\mathcal{U}$. As $\mathcal{D}\subset \mathcal{E}$%
, it follows that $\mathcal{U}\in n\mathcal{E}$, and therefore $n\mathcal{E}%
\prec n\mathcal{D}$. Note also that $n\mathcal{O}=\mathcal{O}$, for every $%
n\in \mathbb{N}^{\ast }$, since for each $\mathcal{U}\in \mathcal{O}$ there
is $\mathcal{V}\in \mathcal{O}$ such that $\mathcal{V}\leqslant \tfrac{1}{%
2^{n}}\mathcal{U}$, that is, $\mathcal{U}\in n\mathcal{O}$.

We now introduce the following notion of convergence in $\mathcal{P}\left( 
\mathcal{O}\right) $.

\begin{definition}
\label{Convergence}We say that a net $\left( \mathcal{E}_{\lambda }\right) $
in $\mathcal{P}\left( \mathcal{O}\right) $ \textbf{converges} to $\mathcal{O}
$, written $\mathcal{E}_{\lambda }\rightarrow \mathcal{O}$,\ if for every $%
\mathcal{U}\in \mathcal{O}$ there is a $\lambda _{0}$ such that $\mathcal{U}%
\in \mathcal{E}_{\lambda }$ whenever $\lambda \geq \lambda _{0}$.
\end{definition}

It is easily seen that $\mathcal{D}_{\lambda }\prec \mathcal{E}_{\lambda }$
and $\mathcal{E}_{\lambda }\rightarrow \mathcal{O}$ implies $\mathcal{D}%
_{\lambda }\rightarrow \mathcal{O}$.

\begin{proposition}
\label{P8} If $\mathcal{E}_{\lambda }\rightarrow \mathcal{O}$ then $n%
\mathcal{E}_{\lambda }\rightarrow \mathcal{O}$ for every $n\in \mathbb{N}%
^{\ast }$.
\end{proposition}

\begin{proof}
For a given $\mathcal{U}\in \mathcal{O}$ there is $\mathcal{V}\in \mathcal{O}
$ such that $\mathcal{U}\in n\left\{ \mathcal{V}\right\} $. As $\mathcal{E}%
_{\lambda }\rightarrow \mathcal{O}$, there is $\lambda _{0}$ such that $%
\mathcal{V}\in \mathcal{E}_{\lambda }$ whenever $\lambda \geq \lambda _{0}$.
By order-preserving, $n\mathcal{E}_{\lambda }\prec n\left\{ \mathcal{V}%
\right\} \prec \left\{ \mathcal{U}\right\} $ whenever $\lambda \geq \lambda
_{0}$. Thus $n\mathcal{E}_{\lambda }\rightarrow \mathcal{O}$.
\end{proof}

We often use the auxiliary function $\rho :X\times X\rightarrow \mathcal{P}%
\left( \mathcal{O}\right) $ given by 
\begin{equation*}
\rho \left( x,y\right) =\left\{ \mathcal{U}\in \mathcal{O}:y\in \mathrm{St}%
\left[ x,\mathcal{U}\right] \right\} .
\end{equation*}%
Note that the value $\rho \left( x,y\right) $ is upwards hereditary, that
is, if $\mathcal{U}\leqslant \mathcal{V}$ with $\mathcal{U}\in \rho \left(
x,y\right) $ then $\mathcal{V}\in \rho \left( x,y\right) $. Thus $\rho
\left( x,y\right) \prec n\rho \left( x,y\right) $ for all $x,y\in X$ and $%
n\in \mathbb{N}^{\ast }$. We expand this property in the following.

\begin{proposition}
\label{P1}The binary operation $\rho :X\times X\rightarrow \mathcal{P}\left( 
\mathcal{O}\right) $ satisfies the following properties:

\begin{enumerate}
\item $\rho \left( x,y\right) =\rho \left( y,x\right) $ for all $x,y\in X$.

\item $\mathcal{O}\prec \rho \left( x,y\right) $, for all $x,y\in X$, and $%
\mathcal{O}=\rho \left( x,x\right) $.

\item If $X$ is Hausdorff, $\mathcal{O}=\rho \left( x,y\right) $ if and only
if $x=y$.

\item $\rho \left( x,y\right) \prec 1\left( \rho \left( x,z\right) \cap \rho
\left( z,y\right) \right) $ for all $x,y,z\in X$.

\item $\rho \left( x,y\right) \prec n\left( \rho \left( x,x_{1}\right) \cap
\rho \left( x_{1},x_{2}\right) \cap \ldots \cap \rho \left( x_{n},y\right)
\right) $ for all $x,y,x_{1},...,x_{n}\in X$.
\end{enumerate}
\end{proposition}

\begin{proof}
For proving item $\left( 1\right) $, it is enough to observe that $y\in 
\mathrm{St}\left[ x,\mathcal{U}\right] $ if and only if $x\in \mathrm{St}%
\left[ y,\mathcal{U}\right] $. For proving item $\left( 2\right) $, we have $%
\rho \left( x,y\right) \subset \mathcal{O}$ for all $x,y\in X$, hence $%
\mathcal{O}\prec \rho \left( x,y\right) $. Since $x\in \mathrm{St}\left[ x,%
\mathcal{U}\right] $ for all $\mathcal{U}\in \mathcal{O}$, we have $\rho
\left( x,x\right) =\mathcal{O}$. For item $\left( 3\right) $, suppose that $%
X $ is Hausdorff and $\rho \left( x,y\right) =\mathcal{O}$. Then $y\in 
\mathrm{St}\left[ x,\mathcal{U}\right] $ for every $\mathcal{U}\in \mathcal{O%
}$. Since the collection $\left\{ \mathrm{St}\left[ x,\mathcal{U}\right] :%
\mathcal{U}\in \mathcal{O}\right\} $ is a neighborhood base at $x$, it
follows that $y=x$, by Hausdorffness. Note that item $\left( 4\right) $ is a
particular case of item $\left( 5\right) $. Then we prove item $\left(
5\right) $. For indeed, if $n\left( \rho \left( x,x_{1}\right) \cap \rho
\left( x_{1},x_{2}\right) \cap \ldots \cap \rho \left( x_{n},y\right)
\right) =\emptyset $ then the inequality is obvious. Suppose that $n\left(
\rho \left( x,x_{1}\right) \cap \rho \left( x_{1},x_{2}\right) \cap \ldots
\cap \rho \left( x_{n},y\right) \right) \neq \emptyset $ and take $\mathcal{U%
}\in n\left( \rho \left( x,x_{1}\right) \cap \rho \left( x_{1},x_{2}\right)
\cap \ldots \cap \rho \left( x_{n},y\right) \right) $. Then there is $%
\mathcal{V}\in \rho \left( x,x_{1}\right) \cap \rho \left(
x_{1},x_{2}\right) \cap \ldots \cap \rho \left( x_{n},y\right) $ with $%
\mathcal{V}\leqslant \frac{1}{2^{n}}\mathcal{U}$, which means the existence
of coverings $\mathcal{U}_{1},...,\mathcal{U}_{n-1}\in \mathcal{O}$ such
that 
\begin{eqnarray*}
\mathcal{V} &\leqslant &\frac{1}{2}\mathcal{U}_{n-1},~\mathcal{U}%
_{n-1}\leqslant \frac{1}{2}\mathcal{U}_{n-2},~...~,~\mathcal{U}_{1}\leqslant 
\frac{1}{2}\mathcal{U}\text{,} \\
x,x_{2} &\in &\mathrm{St}\left[ x_{1},\mathcal{V}\right] ,~x_{1},x_{3}\in 
\mathrm{St}\left[ x_{2},\mathcal{V}\right] ,~...~,~x_{n-1},y\in \mathrm{St}%
\left[ x_{n},\mathcal{V}\right] .
\end{eqnarray*}%
As $\mathcal{V}\leqslant \frac{1}{2}\mathcal{U}_{n-1}$, it follows that 
\begin{equation*}
x,x_{3}\in \mathrm{St}\left[ x_{2},\mathcal{U}_{n-1}\right] ,~x_{2},x_{4}\in 
\mathrm{St}\left[ x_{3},\mathcal{U}_{n-1}\right] ,~...~,~x_{n-2},y\in 
\mathrm{St}\left[ x_{n-1},\mathcal{V}\right] .
\end{equation*}%
By using successively the relations $\mathcal{U}_{n-i}\leqslant \frac{1}{2}%
\mathcal{U}_{n-i-1}$ and $\mathcal{U}_{1}\leqslant \frac{1}{2}\mathcal{U}$
we obtain $y\in \mathrm{St}\left[ x,\mathcal{U}\right] $. Hence $\mathcal{U}%
\in \rho \left( x,y\right) $, and therefore 
\begin{equation*}
\rho \left( x,y\right) \supset n\left( \rho \left( x,x_{1}\right) \cap \rho
\left( x_{1},x_{2}\right) \cap \ldots \cap \rho \left( x_{n},y\right)
\right) .
\end{equation*}
\end{proof}

We can use $\rho $ to characterize convergence, as the following.

\begin{proposition}
\label{P10} Let $\left( x_{\lambda }\right) $ be a net in $X$. Then $%
x_{\lambda }\rightarrow x$ if and only if $\rho \left( x_{\lambda },x\right)
\rightarrow \mathcal{O}$.
\end{proposition}

\begin{proof}
Suppose that $x_{\lambda }\rightarrow x$ and let $\mathcal{U}\in \mathcal{O}$%
. There is $\lambda _{0}$ such that $x_{\lambda }\in \mathrm{St}\left[ x,%
\mathcal{U}\right] $ for all $\lambda \geq \lambda _{0}$, which means that $%
\mathcal{U}\in \rho \left( x_{\lambda },x\right) $ whenever $\lambda \geq
\lambda _{0}$. Thus $\rho \left( x_{\lambda },x\right) \rightarrow \mathcal{O%
}$. As to the converse, suppose that $\rho \left( x_{\lambda },x\right) $
converges to $\mathcal{O}$. For each $\mathcal{U}\in \mathcal{O}$, there is $%
\lambda _{0}$ such that $\mathcal{U}\in \rho \left( x_{\lambda },x\right) $
whenever $\lambda \geq \lambda _{0}$. Hence $x_{\lambda }\in \mathrm{St}%
\left[ x,\mathcal{U}\right] $ for all $\lambda \geq \lambda _{0}$, and
therefore $x_{\lambda }\rightarrow x$.
\end{proof}

We now introduce the notions of bounded set and diameter.

\begin{definition}
A nonempty subset $Y\subset X$ is called\ \textbf{bounded} with respect to $%
\mathcal{O}$ if there is some $\mathcal{U}\in \mathcal{O}$ such that $%
\mathcal{U}\in \rho \left( x,y\right) $ for all $x,y\in Y$
\end{definition}

\begin{definition}
Let $Y\subset X$ be a nonempty set. The \textbf{diameter} of $Y$ is the set $%
\mathrm{D}\left( Y\right) \in \mathcal{P}\left( \mathcal{O}\right) $ defined
as 
\begin{equation*}
\mathrm{D}\left( Y\right) =\bigcap\limits_{x,y\in Y}\rho \left( x,y\right) .
\end{equation*}
\end{definition}

If $Y\subset X$ is a bounded set then $\mathrm{D}\left( Y\right) \neq
\emptyset $. It is easily seen that $\mathcal{U}\in \mathrm{D}\left(
Y\right) $ if and only if $Y$ is bounded by $\mathcal{U}$. In the following
we present some properties of diameter.

\begin{proposition}
\label{P9}Let $A,B\subset X$ be nonempty subsets. The following properties
hold:

\begin{enumerate}
\item $\rho \left( x,y\right) \prec \mathrm{D}\left( A\right) $ for all $%
x,y\in A$.

\item $\mathrm{D}\left( A\right) \prec \mathrm{D}\left( B\right) $ if $%
A\subset B$.

\item $\mathrm{D}\left( A\right) \prec \mathrm{D}\left( \mathrm{cls}\left(
A\right) \right) \prec 2\mathrm{D}\left( A\right) .$
\end{enumerate}
\end{proposition}

\begin{proof}
Items $\left( 1\right) $ and $\left( 2\right) $ follow immediately by
definition. Then we prove item $\left( 3\right) $. Since $A\subset \mathrm{%
cls}\left( A\right) $, the first inequality $\mathrm{D}\left( A\right) \prec 
\mathrm{D}\left( \mathrm{cls}\left( A\right) \right) $ is clear. Let $%
\mathcal{U}\in 2\mathrm{D}\left( A\right) $ and $\bar{x},\bar{y}\in \mathrm{%
cls}\left( A\right) $. Then there is $\mathcal{V}\in \mathrm{D}\left(
A\right) $ such that $\mathcal{V}\leqslant \frac{1}{2^{2}}\mathcal{U}$ and
there are $x,y\in A$ such that $x\in \mathrm{St}\left[ \bar{x},\mathcal{V}%
\right] $ and $y\in \mathrm{St}\left[ \bar{y},\mathcal{V}\right] $. As $%
\mathcal{V}\in \rho \left( x,y\right) $, it follows that $\mathcal{V}\in
\rho \left( x,\bar{x}\right) \cap \rho \left( y,\bar{y}\right) \cap \rho
\left( x,y\right) $, hence $\mathcal{U}\in 2\left( \rho \left( x,\bar{x}%
\right) \cap \rho \left( y,\bar{y}\right) \cap \rho \left( x,y\right)
\right) $. Since $\rho \left( \bar{x},\bar{y}\right) \prec 2\left( \rho
\left( x,\bar{x}\right) \cap \rho \left( y,\bar{y}\right) \cap \rho \left(
x,y\right) \right) $, we have $\mathcal{U}\in \rho \left( \bar{x},\bar{y}%
\right) $. Thus $2\mathrm{D}\left( A\right) \subset \rho \left( \bar{x},\bar{%
y}\right) $ for arbitraries $\bar{x},\bar{y}\in \mathrm{cls}\left( A\right) $%
, which means the inequality $\mathrm{D}\left( \mathrm{cls}\left( A\right)
\right) \prec 2\mathrm{D}\left( A\right) $.
\end{proof}

\section{\label{SectionCompletespace}Complete admissible spaces}

In this section we regard the concept of completeness by using the
admissible structure. We extend some classical theorems which were not
previously discussed in the general setting of uniform spaces. Throughout,
there is a fixed admissible space $X$ endowed with an admissible family $%
\mathcal{O}$ of open coverings of $X$.

The following definition approaches the notion of Cauchy net in uniform
spaces.

\begin{definition}
A net $\left( x_{\lambda }\right) _{\lambda \in \Lambda }$ in $X$ is $%
\mathcal{O}$\textbf{-Cauchy} (or just \textbf{Cauchy}) if for each $\mathcal{%
U}\in \mathcal{O}$ there is some $\lambda _{0}\in \Lambda $ such that $%
x_{\lambda _{1}}$ and $x_{\lambda _{2}}$ lie together in some element of $%
\mathcal{U}$, that is, $x_{\lambda _{1}}\in \mathrm{St}\left[ x_{\lambda
_{2}},\mathcal{U}\right] $, whenever $\lambda _{1},\lambda _{2}\geqslant
\lambda _{0}$.
\end{definition}

Note that a net $\left( x_{\lambda }\right) _{\lambda \in \Lambda }$ is
Cauchy if and only if for each $\mathcal{U}\in \mathcal{O}$ there is some $%
\lambda _{0}\in \Lambda $ such that $\mathcal{U}\in \rho \left( x_{\lambda
_{1}},x_{\lambda _{2}}\right) $ whenever $\lambda _{1},\lambda _{2}\geqslant
\lambda _{0}$.

It is well-known that every convergent net is Cauchy. However, a Cauchy net
need not be convergent, unless it admits convergent subnet, as the following.

\begin{proposition}
\label{Subnet} A Cauchy net that admits a convergent subnet is itself
convergent (with the same limit).
\end{proposition}

\begin{proof}
Let $\left( x_{\lambda }\right) _{\lambda \in \Lambda }$ be a Cauchy net in $%
X$ and $\left( x_{\lambda _{\eta }}\right) $ a convergent subnet with $%
x_{\lambda _{\eta }}\rightarrow x$ in $X$. For a given $\mathcal{U}\in 
\mathcal{O}$, take $\mathcal{V}\in \mathcal{O}$ such that $\mathcal{V}%
\leqslant \frac{1}{2}\mathcal{U}$. There is $\eta _{0}$ such that $\eta \geq
\eta _{0}$ implies $x_{\lambda _{\eta }}\in \mathrm{St}\left[ x,\mathcal{V}%
\right] $. Moreover, there is $\lambda _{0}$ such that $\mathcal{V}\in \rho
\left( x_{\lambda _{1}},x_{\lambda _{2}}\right) $ whenever $\lambda
_{1},\lambda _{2}\geqslant \lambda _{0}$. Take $\lambda _{0}^{\prime }$ such
that $\lambda _{0}^{\prime }\geq \lambda _{0}$ and $\lambda _{0}^{\prime
}\geq \lambda _{\eta _{0}}$, and then fix some $\lambda _{\eta }\geq \lambda
_{0}^{\prime }$. For $\lambda \geq \lambda _{0}^{\prime }$, we have $%
x_{\lambda _{\eta }}\in \mathrm{St}\left[ x,\mathcal{V}\right] \cap \mathrm{%
St}\left[ x_{\lambda },\mathcal{V}\right] $, which implies $x_{\lambda }\in 
\mathrm{St}\left[ x,\mathcal{U}\right] $. Therefore $x_{\lambda }\rightarrow
x$.
\end{proof}

\begin{definition}
If every Cauchy net in the admissible space $X$ converges then $X$ is called
a \textbf{complete admissible space}.
\end{definition}

The uniform local compactness is a criterion to find complete admissible
spaces, as the following.

\begin{definition}
The admissible space $X\,$is \textbf{uniformly locally compact} if there is
some $\mathcal{U}\in \mathcal{O}$ such that $\mathrm{cls}\left( \mathrm{St}%
\left[ x,\mathcal{U}\right] \right) $ is compact for every $x\in X$.
\end{definition}

Note that a uniformly locally compact admissible space is locally compact.
The following extends a well-known theorem of metric spaces.

\begin{proposition}
\label{Uniformlocalcompact} If the admissible space $X\,$is uniformly
locally compact then it is complete. Consequently, if all bounded and closed
subsets of $X\,$are compact then $X$ is locally compact and complete.
\end{proposition}

\begin{proof}
Take $\mathcal{U}\in \mathcal{O}$ such that $\mathrm{cls}\left( \mathrm{St}%
\left[ x,\mathcal{U}\right] \right) $ is compact for every $x\in X$ and let $%
\left( x_{\lambda }\right) $ be a Cauchy net in $X$. Then there is $\lambda
_{0}$ such that $\lambda _{1},\lambda _{2}\geq \lambda _{0}$ implies $%
\mathcal{U}\in \rho \left( x_{\lambda _{1}},x_{\lambda _{2}}\right) $. In
particular, $x_{\lambda }\in \mathrm{St}\left[ x_{\lambda _{0}},\mathcal{U}%
\right] $ for all $\lambda \geq \lambda _{0}$. Then the subset $\left\{
x_{\lambda }\right\} _{\lambda \geq \lambda _{0}}$ is contained in the
compact set $\mathrm{cls}\left( \mathrm{St}\left[ x_{\lambda _{0}},\mathcal{U%
}\right] \right) $. Hence we may assume that $\left( x_{\lambda }\right)
_{\lambda \geq \lambda _{0}}$ converges, and therefore $\left( x_{\lambda
}\right) _{\lambda \in \Lambda }$ converges, by Proposition \ref{Subnet}.
Thus $X$ is complete. Now, suppose that all bounded and closed subsets of $X$
are compact. Take $\mathcal{U}\in \mathcal{O}$ and $\mathcal{V}\in \mathcal{O%
}$ with $\mathcal{V}\leqslant \frac{1}{4}\mathcal{U}$. We claim that $%
\mathrm{cls}\left( \mathrm{St}\left[ x,\mathcal{V}\right] \right) $ is
bounded for every $x\in X$. Indeed, take $\mathcal{V}^{\prime }\in \mathcal{O%
}$ with $\mathcal{V}\leqslant \frac{1}{2}\mathcal{V}^{\prime }$ and $%
\mathcal{V}^{\prime }\leqslant \frac{1}{2}\mathcal{U}$. If $y\in \mathrm{cls}%
\left( \mathrm{St}\left[ x,\mathcal{V}\right] \right) $ then $\mathrm{St}%
\left[ x,\mathcal{V}\right] \cap \mathrm{St}\left[ y,\mathcal{V}\right] \neq
\emptyset $. Hence $\mathcal{V}^{\prime }\in \rho \left( x,y\right) $. For $%
y,z\in \mathrm{cls}\left( \mathrm{St}\left[ x,\mathcal{V}\right] \right) $,
it follows that $\mathcal{V}^{\prime }\in \rho \left( x,y\right) \cap \rho
\left( x,z\right) $, and hence $\mathcal{U}\in \rho \left( y,z\right) $.
Thus $\mathrm{cls}\left( \mathrm{St}\left[ x,\mathcal{V}\right] \right) $ is
bounded and closed. By hypothesis, $\mathrm{cls}\left( \mathrm{St}\left[ x,%
\mathcal{V}\right] \right) $ is compact for every $x\in X$, hence $X$ is
uniformly locally compact. By the first part of the proof, $X$ is complete
and locally compact.
\end{proof}

The following theorem generalizes the classical \emph{Cantor intersection
theorem} from metric spaces. A decreasing net $\left( F_{\lambda }\right) $
of nonempty closed sets of $X$ means that $F_{\lambda _{1}}\subset
F_{\lambda _{2}}$ whenever $\lambda _{1}\geq \lambda _{2}$.

\begin{theorem}[Cantor theorem]
\label{T6} The admissible space $X$ is complete if and only if every
decreasing net $\left( F_{\lambda }\right) $ of nonempty closed sets of $X$,
with $\mathrm{D}\left( F_{\lambda }\right) \rightarrow \mathcal{O}$, has
nonempty intersection. If $X$ is Hausdorff, this intersection consists of a
single point.
\end{theorem}

\begin{proof}
Suppose that $X$ is complete and let $\left( F_{\lambda }\right) $ be a
decreasing net of nonempty closed sets of $X$ with $\mathrm{D}\left(
F_{\lambda }\right) \rightarrow \mathcal{O}$. For each $\lambda $, take $%
x_{\lambda }\in F_{\lambda }$. For a given $\mathcal{U}\in \mathcal{O}$,
there is $\lambda _{0}$ such that $\lambda \geq \lambda _{0}$ implies $%
\mathcal{U}\in \mathrm{D}\left( F_{\lambda }\right) $. For $\lambda
_{1},\lambda _{2}\geq \lambda _{0}$, we have $F_{\lambda _{1}}\cup
F_{\lambda _{2}}\subset F_{\lambda _{0}}$, and then $x_{\lambda
_{1}},x_{\lambda _{2}}\in F_{\lambda _{0}}$. As $\mathcal{U}\in \mathrm{D}%
\left( F_{\lambda _{0}}\right) $, it follows that $\mathcal{U}\in \rho
\left( x_{\lambda _{1}},x_{\lambda _{2}}\right) $ whenever $\lambda
_{1},\lambda _{2}\geq \lambda _{0}$. Hence $\left( x_{\lambda }\right) $ is
a Cauchy net. By completeness, $\left( x_{\lambda }\right) $ converges to
some point $x$. Now, for each $\lambda $, the subnet $\left( x_{\lambda
^{\prime }}\right) _{\lambda ^{\prime }\geq \lambda }$ also converges to $x$%
. As $x_{\lambda ^{\prime }}\in F_{\lambda ^{\prime }}\subset F_{\lambda }$,
for all $\lambda ^{\prime }\geq \lambda $, and $F_{\lambda }$ is closed, it
follows that $x\in F_{\lambda }$. Therefore $x\in \bigcap\limits_{\lambda
}F_{\lambda }$. As to the converse, suppose that every decreasing net $%
\left( F_{\lambda }\right) $ of nonempty closed sets of $X$, with $\mathrm{D}%
\left( F_{\lambda }\right) \rightarrow \mathcal{O}$, has nonempty
intersection. Let $\left( x_{\lambda }\right) $ be a Cauchy net. For each $%
\lambda $, define the set $H_{\lambda }=\left\{ x_{\lambda ^{\prime
}}:\lambda ^{\prime }\geq \lambda \right\} $. Then $H_{\lambda }$ is a
nonempty subset of $X$ and $H_{\lambda _{1}}\subset H_{\lambda _{2}}$
whenever $\lambda _{1}\geq \lambda _{2}$. We claim that $\mathrm{D}\left(
H_{\lambda }\right) \rightarrow \mathcal{O}$. In fact, for a given $\mathcal{%
U}\in \mathcal{O}$ there is $\lambda _{0}$ such that $\mathcal{U}\in \rho
\left( x_{\lambda _{1}},x_{\lambda _{2}}\right) $ whenever $\lambda
_{1},\lambda _{2}\geq \lambda _{0}$. If $x_{\lambda _{1}},x_{\lambda
_{2}}\in H_{\lambda }$, with $\lambda \geq \lambda _{0}$, we have $\lambda
_{1},\lambda _{2}\geq \lambda \geq \lambda _{0}$ and then $\mathcal{U}\in
\rho \left( x_{\lambda _{1}},x_{\lambda _{2}}\right) $. Hence $\mathcal{U}%
\in \mathrm{D}\left( H_{\lambda }\right) $ whenever $\lambda \geq \lambda
_{0}$, and therefore $\mathrm{D}\left( H_{\lambda }\right) \rightarrow 
\mathcal{O}$. Now, by Proposition \ref{P8}, we have $2\mathrm{D}\left(
H_{\lambda }\right) \rightarrow \mathcal{O}$. By Proposition \ref{P9}, we
have $\mathrm{D}\left( H_{\lambda }\right) \prec \mathrm{D}\left( \mathrm{cls%
}\left( H_{\lambda }\right) \right) \prec 2\mathrm{D}\left( H_{\lambda
}\right) $. Thus $\mathrm{D}\left( \mathrm{cls}\left( H_{\lambda }\right)
\right) \rightarrow \mathcal{O}$. Since $\mathrm{cls}\left( H_{\lambda
_{1}}\right) \subset \mathrm{cls}\left( H_{\lambda _{2}}\right) $ whenever $%
\lambda _{1}\geq \lambda _{2}$, we have obtain a decreasing net $\left( 
\mathrm{cls}\left( H_{\lambda }\right) \right) $ of nonempty closed sets of $%
X$ with $\mathrm{D}\left( \mathrm{cls}\left( H_{\lambda }\right) \right)
\rightarrow \mathcal{O}$. By hypothesis, the intersection $%
\bigcap\limits_{\lambda }\mathrm{cls}\left( H_{\lambda }\right) $ is
nonempty. Then we can take a point $x\in \bigcap\limits_{\lambda }\mathrm{cls%
}\left( H_{\lambda }\right) $. We claim that $x_{\lambda }\rightarrow x$. By
Proposition \ref{P10}, it means to prove that $\rho \left( x_{\lambda
},x\right) \rightarrow \mathcal{O}$. For a given $\mathcal{U}\in \mathcal{O}$%
, there is $\lambda _{0}$ such that $\lambda \geq \lambda _{0}$ implies $%
\mathcal{U}\in \mathrm{D}\left( \mathrm{cls}\left( H_{\lambda }\right)
\right) $. As $x_{\lambda },x\in \mathrm{cls}\left( H_{\lambda }\right) $,
it follows that $\mathcal{U}\in \rho \left( x_{\lambda },x\right) $ whenever 
$\lambda \geq \lambda _{0}$. Hence $\rho \left( x_{\lambda },x\right)
\rightarrow \mathcal{O}$. Therefore the Cauchy net $\left( x_{\lambda
}\right) $ is convergent and $X$ is a complete admissible space. We now show
the second part of the theorem. Assume that $X$ is Hausdorff. Suppose that $%
X $ is complete and let $\left( F_{\lambda }\right) $ be a decreasing net of
nonempty closed sets of $X$ with $\mathrm{D}\left( F_{\lambda }\right)
\rightarrow \mathcal{O}$. By the first part of the proof, there is a point $%
x\in \bigcap\limits_{\lambda }F_{\lambda }$. If $y\in
\bigcap\limits_{\lambda }F_{\lambda }$ and $\mathcal{U}\in \mathcal{O}$ then 
$\mathcal{U}\in \rho \left( x,y\right) $, because $\mathrm{D}\left(
F_{\lambda }\right) \rightarrow \mathcal{O}$ and $x,y\in F_{\lambda }$ for
all $\lambda $. This means that $\rho \left( x,y\right) =\mathcal{O}$, and
therefore $x=y$.
\end{proof}

Complete admissible spaces seem quite a bit like compact spaces. The
difference between completeness and compactness is the total boundedness.
Recall that a covering uniformity on $X$ is totally bounded if it has a base
consisting of finite covers. If $X$ is equipped with a totally bounded
covering uniformity, it is called a \emph{totally bounded uniform space} (%
\cite[Definition 39.7]{Will}). Inspired by the definition of totally bounded
metric space, we shall provide an approach of total boundedness by means of
the notion of diameter.

\begin{definition}
\label{Totallybounded}The admissible space $X$ is said to be \textbf{totally
bounded} with respect to $\mathcal{O}$ if for each $\mathcal{U}\in \mathcal{O%
}$ there is a finite cover $X=X_{1}\cup \ldots \cup X_{n}$ by sets with
diameter containing $\mathcal{U}$.
\end{definition}

In the end of this section we show the connection of Definition \ref%
{Totallybounded} with the current definition of totally bounded uniform
space.

\begin{proposition}
\label{P18} The admissible space $X$ is totally bounded if and only if for
each $\mathcal{U}\in \mathcal{O}$ there is a finite sequence $%
x_{1},...,x_{n} $ of points in $X$ such that $X=\bigcup\limits_{i=1}^{n}%
\mathrm{St}\left[ x_{i},\mathcal{U}\right] $.
\end{proposition}

\begin{proof}
Suppose that $X$ is totally bounded and let $\mathcal{U}\in \mathcal{O}$.
Then $X=X_{1}\cup \ldots \cup X_{n}$ with $\mathcal{U}\in \mathrm{D}\left(
X_{i}\right) $. By choosing $x_{i}\in X_{i}$, we have $X_{i}\subset \mathrm{S%
}\left[ x_{i},\mathcal{U}\right] $. Thus $X=\bigcup\limits_{i=1}^{n}\mathrm{%
St}\left[ x_{i},\mathcal{U}\right] $. As to the converse, for a given $%
\mathcal{U}\in \mathcal{O}$ take $\mathcal{V}\in \mathcal{O}$ with $\mathcal{%
V}\leqslant \frac{1}{2}\mathcal{U}$. Then there is a finite sequence $%
x_{1},...,x_{n}$ such that $X=\bigcup\limits_{i=1}^{n}\mathrm{St}\left[
x_{i},\mathcal{V}\right] $. For $x,y\in \mathrm{St}\left[ x_{i},\mathcal{V}%
\right] $, we have $\mathcal{U}\in \rho \left( x,y\right) $, and then $%
\mathcal{U}\in \mathrm{D}\left( \mathrm{St}\left[ x_{i},\mathcal{V}\right]
\right) $. Therefore $X$ is totally bounded.
\end{proof}

In general, we say that a subset $Y\subset X$ is totally bounded if for each 
$\mathcal{U}\in \mathcal{O}$ there is a finite sequence $x_{1},...,x_{n}$
such that $Y\subset \bigcup\limits_{i=1}^{n}\mathrm{St}\left[ x_{i},\mathcal{%
V}\right] $. The following theorems are proved in the setting of uniform
spaces by considering the usual notion of total boundedness. Here we
reproduce them by considering Definition \ref{Totallybounded}.

\begin{proposition}
\label{Totalbounded} The admissible space $X$ is totally bounded if and only
if each net in $X$ has a Cauchy subnet.
\end{proposition}

\begin{proof}
Suppose that $X$ is totally bounded and let $\left( x_{\lambda }\right) $ be
a net in $X$. For a given $\mathcal{U}\in \mathcal{O}$, there is a set $X_{%
\mathcal{U}}\subset X$ with $\mathcal{U}\in \mathrm{D}\left( X_{\mathcal{U}%
}\right) $ and such that $\left( x_{\lambda }\right) $ is frequently in $X_{%
\mathcal{U}}$, that is, for each $\lambda $ there is $\lambda ^{\prime }$
with $\lambda ^{\prime }\geq \lambda $ and $x_{\lambda ^{\prime }}\in X_{%
\mathcal{U}}$. Let $\Gamma =\left\{ \left( \lambda ,\mathcal{U}\right) :%
\mathcal{U}\in \mathcal{O}\text{ and }x_{\lambda }\in X_{\mathcal{U}%
}\right\} $ directed by $\left( \lambda _{1},\mathcal{U}_{1}\right) \leq
\left( \lambda _{2},\mathcal{U}_{2}\right) $ iff $\lambda _{1}\leq \lambda
_{2}$ and $X_{\mathcal{U}_{1}}\supset X_{\mathcal{U}_{2}}$. For each $\left(
\lambda ,\mathcal{U}\right) \in \Gamma $ define $x_{\left( \lambda ,\mathcal{%
U}\right) }=x_{\lambda }$. Then $\left( x_{\left( \lambda ,\mathcal{U}%
\right) }\right) $ is a subnet of $\left( x_{\lambda }\right) $. For a given 
$\mathcal{U}\in \mathcal{O}$, take $\lambda $ such that $\left( \lambda ,%
\mathcal{U}\right) \in \Gamma $. If $\left( \lambda _{1},\mathcal{U}%
_{1}\right) ,\left( \lambda _{2},\mathcal{U}_{2}\right) \geq \left( \lambda ,%
\mathcal{U}\right) $ then 
\begin{eqnarray*}
x_{\left( \lambda _{1},\mathcal{U}_{1}\right) } &=&x_{\lambda _{1}}\in X_{%
\mathcal{U}_{1}}\subset X_{\mathcal{U}}, \\
x_{\left( \lambda _{2},\mathcal{U}_{2}\right) } &=&x_{\lambda _{2}}\in X_{%
\mathcal{U}_{2}}\subset X_{\mathcal{U}},
\end{eqnarray*}%
hence $\mathcal{U}\in \rho \left( x_{\left( \lambda _{1},\mathcal{U}%
_{1}\right) },x_{\left( \lambda _{2},\mathcal{U}_{2}\right) }\right) $.
Therefore $\left( x_{\left( \lambda ,\mathcal{U}\right) }\right) $ is a
Cauchy subnet of $\left( x_{\lambda }\right) $. As to the converse, suppose
that each net in $X$ has a Cauchy subnet. By Proposition \ref{P18}, if $X$
is not totally bounded then there is $\mathcal{U}\in \mathcal{O}$ such that
the covering $X=\bigcup\limits_{x\in X}\mathrm{St}\left[ x,\mathcal{U}\right]
$ does not admit finite subcovering. Then we can construct by induction a
sequence $\left( x_{n}\right) _{n\in \mathbb{N}}$ in $X$ such that $%
x_{n+1}\notin \bigcup\limits_{i=1}^{n}\mathrm{St}\left[ x_{i},\mathcal{U}%
\right] $. By hypothesis, there is a Cauchy subnet $\left( x_{n_{\lambda
}}\right) $ of $\left( x_{n}\right) $. Hence there is some $\lambda _{0}$
such that $\mathcal{U}\in \rho \left( x_{n_{\lambda }},x_{n_{\lambda
^{\prime }}}\right) $ whenever $\lambda ,\lambda ^{\prime }\geq \lambda _{0}$%
. By taking $\lambda ,\lambda ^{\prime }\geq \lambda _{0}$ with $n_{\lambda
}>n_{\lambda ^{\prime }}$, it follows that $x_{n_{\lambda }}\in \mathrm{St}%
\left[ x_{n_{\lambda ^{\prime }}},\mathcal{U}\right] $, and then $x_{\left(
n_{\lambda }-1\right) +1}\in \bigcup\limits_{i=1}^{n_{\lambda }-1}\mathrm{St}%
\left[ x_{i},\mathcal{U}\right] $. This is a contradiction.
\end{proof}

We now link compactness and completeness.

\begin{theorem}
\label{CompactComplete}The admissible space $X$ is compact if and only if it
is complete and totally bounded.
\end{theorem}

\begin{proof}
Assume that $X$ is compact. Then, for every $\mathcal{U}\in \mathcal{O}$,
the open covering $X=\bigcup\limits_{x\in X}\mathrm{St}\left[ x,\mathcal{E}%
\right] $ has a finite subcovering, and hence $X$ is totally bounded.
Moreover, if $\left( x_{\lambda }\right) $ is any Cauchy net in $X$, then $%
\left( x_{\lambda }\right) $ has a convergent subnet in $X$. By Proposition %
\ref{Subnet}, $\left( x_{\lambda }\right) $ is convergent, and therefore $X$
is complete. On the other hand, suppose that $X$ is complete and totally
bounded. By Proposition \ref{Totalbounded}, every net in $X$ has a Cauchy
subnet. By completeness, it follows that every net in $X$ has a convergent
subnet, and therefore $X$ is compact.
\end{proof}

Finally, we turn to show the connection of Definition \ref{Totallybounded}
with the current definition of totally bounded uniform space. We consider
the equivalence between admissible space and uniformizable space stated in 
\cite[Theorem 1]{Richard}. For each $\mathcal{U}\in \mathcal{O}$ we define
the open covering $\mathcal{B}_{\mathcal{U}}=\left\{ \mathrm{St}\left[ x,%
\mathcal{U}\right] :x\in X\right\} $. The family $\mathfrak{B}$ of all open
coverings $\mathcal{B}_{\mathcal{U}}$ generates the covering uniformity $%
\widetilde{\mathfrak{B}}$ on $X$ defined as 
\begin{equation*}
\widetilde{\mathfrak{B}}=\left\{ \mathcal{V}:\mathcal{V}\text{ is an open
covering of }X\text{ and }\mathcal{B}_{\mathcal{U}}\leqslant \mathcal{V}%
\text{ for some }\mathcal{U}\in \mathcal{O}\right\} .
\end{equation*}

\begin{proposition}
The admissible space $X$ is totally bounded if and only if the covering
uniformity $\widetilde{\mathfrak{B}}$ is totally bounded.
\end{proposition}

\begin{proof}
Suppose that $X$ is totally bounded by and let $\mathcal{U}\in \mathcal{O}$.
By Proposition \ref{P18}, there is a finite sequence $x_{1},...,x_{n}$ of
points in $X$ such that $X=\bigcup\limits_{i=1}^{n}\mathrm{St}\left[ x_{i},%
\mathcal{U}\right] $. Hence $\mathcal{B}_{\mathcal{U}}^{\prime }=\left\{ 
\mathrm{St}\left[ x_{i},\mathcal{U}\right] :i=1,...,n\right\} $ is a finite
subcovering of $\mathcal{B}_{\mathcal{U}}$. Set $\mathfrak{B}^{\prime
}=\left\{ \mathcal{B}_{\mathcal{U}}^{\prime }:\mathcal{U}\in \mathcal{O}%
\right\} $. Since $\mathfrak{B}$ is a base for the covering uniformity $%
\widetilde{\mathfrak{B}}$ on $X$, it follows that $\mathfrak{B}^{\prime }$
is a base consisting of finite covers. Therefore $\widetilde{\mathfrak{B}}$
is totally bounded. As to the converse, suppose that $\widetilde{\mathfrak{B}%
}$ has a base consisting of finite covers and let $\mathcal{U}\in \mathcal{O}
$. Then there is a finite subcovering $\left\{ \mathrm{St}\left[ x_{i},%
\mathcal{U}\right] :i=1,...,n\right\} $ of $\mathcal{B}_{\mathcal{U}}$. It
follows that $X=\bigcup\limits_{i=1}^{n}\mathrm{St}\left[ x_{i},\mathcal{U}%
\right] $, and therefore $X$ is totally bounded, by Proposition \ref{P18}.
\end{proof}

\section{Cantor--Kuratowski theorem}

In this section present the main theorem of the paper. We reproduce the
notions of measure of noncompactness to provide a generalization of the
Cantor--Kuratowski intersection theorem. Throughout, there is a fixed
admissible space $X$ endowed with an admissible family $\mathcal{O}$ of open
coverings of $X$.

\begin{definition}
Let $Y\subset X$ be a nonempty set. The \textbf{star measure of
noncompactness} of $Y$ is the set $\alpha \left( Y\right) \in \mathcal{P}%
\left( \mathcal{O}\right) $ defined as 
\begin{equation*}
\alpha \left( Y\right) =\left\{ \mathcal{U}\in \mathcal{O}:Y\text{ admits a
finite cover }Y\subset \bigcup\limits_{i=1}^{n}\mathrm{St}\left[ x_{i},%
\mathcal{U}\right] \right\} ;
\end{equation*}%
the \textbf{diameter measure of noncompactness} of $Y$ is the set $\gamma
\left( Y\right) \in \mathcal{P}\left( \mathcal{O}\right) $ defined as 
\begin{equation*}
\gamma \left( Y\right) =\left\{ \mathcal{U}\in \mathcal{O}:Y\text{ admits a
finite cover }Y\subset \bigcup\limits_{i=1}^{n}X_{i}\text{ with }\mathcal{U}%
\in \mathrm{D}\left( X_{i}\right) \right\} .
\end{equation*}
\end{definition}

The diameter measure of noncompactness generalizes the \emph{Kuratowski
measure of noncompactness} of metric spaces. If $Y\subset X$ is a bounded
set then both the sets $\alpha \left( Y\right) $ and $\gamma \left( Y\right) 
$ are nonempty. In the following we present some properties of measure of
noncompactness.

\begin{proposition}
\label{P19} The following properties hold:

\begin{enumerate}
\item $\alpha \left( Y\right) \prec \gamma \left( Y\right) \prec 1\alpha
\left( Y\right) .$

\item $\alpha \left( Y\right) \prec \mathrm{D}\left( Y\right) .$

\item $\alpha \left( Y\right) \prec \alpha \left( Z\right) $ if $Y\subset Z$.

\item $\alpha \left( Y\cup Z\right) =\alpha \left( Y\right) \cap \alpha
\left( Y\right) .$

\item $\alpha \left( Y\right) \prec \alpha \left( \mathrm{cls}\left(
Y\right) \right) \prec 1\alpha \left( Y\right) .$
\end{enumerate}
\end{proposition}

\begin{proof}
$\left( 1\right) $ If $\mathcal{U}\in \gamma \left( Y\right) $ then $Y$
admits a finite cover $Y\subset \bigcup\limits_{i=1}^{n}X_{i}$ with $%
\mathcal{U}\in \mathrm{D}\left( X_{i}\right) $. For each $i=1,...,n$, pick $%
x_{i}\in X_{i}$. We have $X_{i}\subset \mathrm{St}\left[ x_{i},\mathcal{U}%
\right] $. It follows that $Y\subset \bigcup\limits_{i=1}^{n}\mathrm{St}%
\left[ x_{i},\mathcal{U}\right] $, and hence $\mathcal{U}\in \alpha \left(
Y\right) $. Therefore $\gamma \left( Y\right) \subset \alpha \left( Y\right) 
$, that is, $\alpha \left( Y\right) \prec \gamma \left( Y\right) $. Now
suppose that $\mathcal{U}\in 1\alpha \left( Y\right) $. Then there is $%
\mathcal{V}\in \alpha \left( Y\right) $ such that $\mathcal{V}\leqslant 
\frac{1}{2}\mathcal{U}$. This means that $Y$ admits a finite cover $Y\subset
\bigcup\limits_{i=1}^{n}\mathrm{St}\left[ x_{i},\mathcal{V}\right] $. For
any pair $x,y\in \mathrm{St}\left[ x_{i},\mathcal{V}\right] $, we have $y\in 
\mathrm{St}\left[ x,\mathcal{U}\right] $, hence $\mathcal{U}\in \mathrm{D}%
\left( \mathrm{St}\left[ x_{i},\mathcal{V}\right] \right) $. Then we have $%
Y\subset \bigcup\limits_{i=1}^{n}\mathrm{St}\left[ x_{i},\mathcal{V}\right] $
with $\mathcal{U}\in \mathrm{D}\left( \mathrm{St}\left[ x_{i},\mathcal{V}%
\right] \right) $, which means that $\mathcal{U}\in \gamma \left( Y\right) $%
. Therefore $\gamma \left( Y\right) \prec 1\alpha \left( Y\right) $.

$\left( 2\right) $ If $\mathcal{U}\in \mathrm{D}\left( Y\right) $ then $y\in 
\mathrm{St}\left[ x,\mathcal{U}\right] $ for all $x,y\in Y$. By choosing $%
x\in Y$, we have $Y\subset \mathrm{St}\left[ x,\mathcal{U}\right] $, hence $%
\mathcal{U}\in \alpha \left( Y\right) $. Thus $\alpha \left( Y\right) \prec 
\mathrm{D}\left( Y\right) $.

$\left( 3\right) $ If $\mathcal{U}\in \alpha \left( Z\right) $ then $Z$
admits a finite cover $Z\subset \bigcup\limits_{i=1}^{n}\mathrm{St}\left[
x_{i},\mathcal{U}\right] $. Since $Y\subset Z$, $Y$ admits a finite cover $%
Y\subset \bigcup\limits_{i=1}^{n}\mathrm{St}\left[ x_{i},\mathcal{U}\right] $%
. Hence $\alpha \left( Y\right) \prec \alpha \left( Z\right) $.

$\left( 4\right) $ If $\mathcal{U}\in \alpha \left( Y\cup Z\right) $ then $%
Y\cup Z$ admits a finite cover $Y\cup Z\subset \bigcup\limits_{i=1}^{n}%
\mathrm{St}\left[ x_{i},\mathcal{U}\right] $. Since $Y\subset Y\cup Z$ and $%
Z\subset Y\cup Z$, it follows that $\mathcal{U}\in \alpha \left( Y\right)
\cap \alpha \left( Z\right) $. On the other hand, if $\mathcal{V}\in \alpha
\left( Y\right) \cap \alpha \left( Z\right) $ then $Y$ admits a finite cover 
$Y\subset \bigcup\limits_{i=1}^{n}\mathrm{St}\left[ x_{i},\mathcal{V}\right] 
$ and $Z$ admits a finite cover $Z\subset \bigcup\limits_{j=1}^{m}\mathrm{St}%
\left[ y_{j},\mathcal{V}\right] $. Hence $Y\cup Z$ admits a finite cover $%
Y\cup Z\subset \left\{ \bigcup\limits_{i=1}^{n}\mathrm{St}\left[ x_{i},%
\mathcal{V}\right] \right\} \cup \left\{ \bigcup\limits_{j=1}^{m}\mathrm{St}%
\left[ y_{j},\mathcal{V}\right] \right\} $. Thus $\mathcal{V}\in \alpha
\left( Y\cup Z\right) $.

$\left( 5\right) $ The inequality $\alpha \left( Y\right) \prec \alpha
\left( \mathrm{cls}\left( Y\right) \right) $ is clear by item $\left(
3\right) $. If $\mathcal{U}\in 1\alpha \left( Z\right) $ then there is $%
\mathcal{V}\in \alpha \left( Y\right) $ such that $\mathcal{V}\leqslant 
\frac{1}{2}\mathcal{U}$. This means that $Y$ admits a finite cover $Y\subset
\bigcup\limits_{i=1}^{n}\mathrm{St}\left[ x_{i},\mathcal{V}\right] $. Then
we have 
\begin{equation*}
\mathrm{cls}\left( Y\right) \subset \bigcup\limits_{i=1}^{n}\mathrm{cls}%
\left( \mathrm{St}\left[ x_{i},\mathcal{V}\right] \right) \subset
\bigcup\limits_{i=1}^{n}\mathrm{St}\left[ x_{i},\mathcal{U}\right]
\end{equation*}%
hence $\mathcal{U}\in \alpha \left( \mathrm{cls}\left( Y\right) \right) $.
Thus $\alpha \left( \mathrm{cls}\left( Y\right) \right) \prec 1\alpha \left(
Y\right) $.
\end{proof}

By item $\left( 1\right) $ of Proposition \ref{P19}, we may consider any
measure of noncompactness for theoretic problems.

In the following we present the main property of measure of noncompactness.

\begin{proposition}
Let $Y\subset X$ be a nonempty set. If $\mathrm{cls}\left( Y\right) $ is
compact then $\alpha \left( Y\right) =\mathcal{O}$. The converse holds if $X$
is a complete admissible space.
\end{proposition}

\begin{proof}
Suppose that $\mathrm{cls}\left( Y\right) $ is compact. For a given $%
\mathcal{U}\in \mathcal{O}$, the open covering $\mathrm{cls}\left( Y\right)
\subset \bigcup\limits_{x\in \mathrm{cls}\left( Y\right) }\mathrm{St}\left[
x,\mathcal{U}\right] \,$\ admits a finite subcovering $\mathrm{cls}\left(
Y\right) \subset \bigcup\limits_{i=1}^{n}\mathrm{St}\left[ x_{i},\mathcal{U}%
\right] $. Hence $\alpha \left( \mathrm{cls}\left( Y\right) \right) =%
\mathcal{O}$. By Proposition \ref{P19}, it follows that $\alpha \left(
Y\right) \prec \alpha \left( \mathrm{cls}\left( Y\right) \right) =\mathcal{O}
$, which means $\alpha \left( Y\right) =\mathcal{O}$. For the converse,
assume that $X$ is complete. If $\alpha \left( Y\right) =\mathcal{O}$ then $%
1\alpha \left( Y\right) =1\mathcal{O}=\mathcal{O}$. By Proposition \ref{P19}
again, $\alpha \left( \mathrm{cls}\left( Y\right) \right) \prec 1\alpha
\left( Y\right) =\mathcal{O}$, hence $\alpha \left( \mathrm{cls}\left(
Y\right) \right) =\mathcal{O}$. Then, for each $\mathcal{U}\in \mathcal{O}$, 
$\mathrm{cls}\left( Y\right) $ admits a finite cover $\mathrm{cls}\left(
Y\right) \subset \bigcup\limits_{i=1}^{n}\mathrm{St}\left[ x_{i},\mathcal{U}%
\right] $. Thus $\mathrm{cls}\left( Y\right) $ is totally bounded and
closed. Since $X$ is complete, it follows that $\mathrm{cls}\left( Y\right) $
is compact, by Theorem \ref{CompactComplete}.
\end{proof}

In particular, if $X$ is a complete admissible space and $Y\subset X$ is a
nonempty closed set then $Y$ is compact if and only if $\alpha \left(
Y\right) =\mathcal{O}$ (or $\gamma \left( Y\right) =\mathcal{O}$).

We now prove the generalization of the Cantor--Kuratowski intersection
theorem.

\begin{theorem}[Cantor--Kuratowski theorem]
\label{T7} The admissible space $X$ is complete if and only if every
decreasing net $\left( F_{\lambda }\right) $ of nonempty bounded closed sets
of $X$, with $\gamma \left( F_{\lambda }\right) \rightarrow \mathcal{O}$,
has nonempty compact intersection.
\end{theorem}

\begin{proof}
Suppose that $X$ is complete and let $\left( F_{\lambda }\right) $ be a
decreasing net of nonempty bounded closed sets of $X$ with $\gamma \left(
F_{\lambda }\right) \rightarrow \mathcal{O}$. For each $\lambda $, take $%
x_{\lambda }\in F_{\lambda }$. For a given $\mathcal{U}\in \mathcal{O}$, we
claim that there is a star $\mathrm{St}\left[ x_{\mathcal{U}},\mathcal{U}%
\right] $ such that the net $\left( x_{\lambda }\right) $ is frequently in $%
\mathrm{St}\left[ x_{\mathcal{U}},\mathcal{U}\right] $, that is, for each $%
\lambda $ there is $\lambda ^{\prime }$ with $\lambda ^{\prime }\geq \lambda 
$ and $x_{\lambda ^{\prime }}\in \mathrm{St}\left[ x_{\mathcal{U}},\mathcal{U%
}\right] $. Indeed, as $\gamma \left( F_{\lambda }\right) \rightarrow 
\mathcal{O}$, there is $\lambda _{0}$ such that $\lambda \geq \lambda _{0}$
implies $\mathcal{U}\in \gamma \left( F_{\lambda }\right) $. Then there are $%
x_{1},...,x_{n}\in X$ such that $F_{\lambda _{0}}\subset
\bigcup\limits_{i=1}^{n}\mathrm{St}\left[ x_{i},\mathcal{U}\right] $. Now,
for each $\lambda $ arbitrary, we can take $\lambda ^{\prime }$ such that $%
\lambda ^{\prime }\geq \lambda $ and $\lambda ^{\prime }\geq \lambda _{0}$,
and then we have 
\begin{equation*}
F_{\lambda ^{\prime }}\subset F_{\lambda }\cap F_{\lambda _{0}}\subset
\bigcup\limits_{i=1}^{n}\mathrm{St}\left[ x_{i},\mathcal{U}\right] .
\end{equation*}%
We may assume that $x_{\lambda _{0}}\in \mathrm{St}\left[ x_{1},\mathcal{U}%
\right] $. If for each $\lambda $ there is $\lambda ^{\prime }\geq \lambda $
and $\lambda ^{\prime }\geq \lambda _{0}$ such that $x_{\lambda ^{\prime
}}\in \mathrm{St}\left[ x_{1},\mathcal{U}\right] $, the claim is proved.
Otherwise, if there is some $\lambda ^{\ast }$ such that $x_{\lambda
^{\prime }}\notin \mathrm{St}\left[ x_{1},\mathcal{U}\right] $ for all $%
\lambda ^{\prime }\geq \lambda ^{\ast }$ and $\lambda ^{\prime }\geq \lambda
_{0}$, we have $x_{\lambda ^{\prime }}\in \bigcup\limits_{i=2}^{n}\mathrm{St}%
\left[ x_{i},\mathcal{U}\right] $ for all $\lambda ^{\prime }$ with $\lambda
^{\prime }\geq \lambda ^{\ast }$ and $\lambda ^{\prime }\geq \lambda _{0}$.
By choosing $\lambda _{1}$ such that $\lambda _{1}\geq \lambda ^{\ast }$ and 
$\lambda _{1}\geq \lambda _{0}$, we have $x_{\lambda }\in
\bigcup\limits_{i=2}^{n}\mathrm{St}\left[ x_{i},\mathcal{U}\right] $
whenever $\lambda \geq \lambda _{1}$. We may assume that $x_{\lambda
_{1}}\in \mathrm{St}\left[ x_{2},\mathcal{U}\right] $. If for each $\lambda $
there is $\lambda ^{\prime }\geq \lambda $ and $\lambda ^{\prime }\geq
\lambda _{1}$ such that $x_{\lambda ^{\prime }}\in \mathrm{St}\left[ x_{2},%
\mathcal{U}\right] $, the claim is proved. Otherwise, we may repeat the
argument to find $\lambda _{2}$ such that $x_{\lambda }\in
\bigcup\limits_{i=3}^{n}\mathrm{St}\left[ x_{i},\mathcal{U}\right] $
whenever $\lambda \geq \lambda _{2}$. Following by induction we can find $%
k<n $ and $\lambda _{k}$ such that $x_{\lambda _{k}}\in \mathrm{St}\left[
x_{k+1},\mathcal{U}\right] $ and for each $\lambda $ there is $\lambda
^{\prime }\geq \lambda $ and $\lambda ^{\prime }\geq \lambda _{k}$ with $%
x_{\lambda ^{\prime }}\in \mathrm{St}\left[ x_{k+1},\mathcal{U}\right] $,
and therefore the claim is proved. We now define the set 
\begin{equation*}
\Gamma =\left\{ \left( \lambda ,\mathcal{U}\right) :\mathcal{U}\in \mathcal{O%
},x_{\lambda }\in \mathrm{St}\left[ x_{\mathcal{U}},\mathcal{U}\right]
\right\}
\end{equation*}%
directed by $\left( \lambda _{1},\mathcal{U}_{1}\right) \leq \left( \lambda
_{2},\mathcal{U}_{2}\right) $ if $\lambda _{1}\leq \lambda _{2}$ and $%
\mathrm{St}\left[ x_{\mathcal{U}_{1}},\mathcal{U}_{1}\right] \supset \mathrm{%
St}\left[ x_{\mathcal{U}_{2}},\mathcal{U}_{2}\right] $. For each $\left(
\lambda ,\mathcal{U}\right) \in \Gamma $ we define $x_{\left( \lambda ,%
\mathcal{U}\right) }=x_{\lambda }$. As $\left( x_{\lambda }\right) $ is
frequently in each $\mathrm{St}\left[ x_{\mathcal{U}},\mathcal{U}\right] $, $%
\left( x_{\left( \lambda ,\mathcal{U}\right) }\right) _{\left( \lambda ,%
\mathcal{U}\right) \in \Gamma }$ is a subnet of $\left( x_{\lambda }\right) $%
. For a given $\mathcal{U}\in \mathcal{O}$, take $\mathcal{V}\in \mathcal{O}$
and $\lambda $ such that $\mathcal{V}\leqslant \frac{1}{2}\mathcal{U}$ and $%
\left( \lambda ,\mathcal{V}\right) \in \Gamma $. If $\left( \lambda _{1},%
\mathcal{U}_{1}\right) ,\left( \lambda _{2},\mathcal{U}_{2}\right) \geq
\left( \lambda ,\mathcal{V}\right) $ then 
\begin{eqnarray*}
x_{\left( \lambda _{1},\mathcal{U}_{1}\right) } &=&x_{\lambda _{1}}\in 
\mathrm{St}\left[ x_{\mathcal{U}_{1}},\mathcal{U}_{1}\right] \subset \mathrm{%
St}\left[ x_{\mathcal{V}},\mathcal{V}\right] , \\
x_{\left( \lambda _{2},\mathcal{U}_{2}\right) } &=&x_{\lambda _{2}}\in 
\mathrm{St}\left[ x_{\mathcal{U}_{2}},\mathcal{U}_{2}\right] \subset \mathrm{%
St}\left[ x_{\mathcal{V}},\mathcal{V}\right] ,
\end{eqnarray*}%
and hence 
\begin{equation*}
\rho \left( x_{\left( \lambda _{1},\mathcal{E}_{1}\right) },x_{\left(
\lambda _{2},\mathcal{E}_{2}\right) }\right) \prec 1\left( \rho \left(
x_{\left( \lambda _{1},\mathcal{E}_{1}\right) },x_{\mathcal{D}}\right) \cap
\rho \left( x_{\mathcal{D}},x_{\left( \lambda _{2},\mathcal{E}_{2}\right)
}\right) \right) \prec 1\left\{ \mathcal{V}\right\} \prec \left\{ \mathcal{U}%
\right\} \text{.}
\end{equation*}%
Therefore $\left( x_{\left( \lambda ,\mathcal{U}\right) }\right) $ is a
Cauchy subnet of $\left( x_{\lambda }\right) $. By completeness, $\left(
x_{\left( \lambda ,\mathcal{U}\right) }\right) $ converges to some point $%
x\in X$. We claim that $x\in \bigcap\limits_{\lambda }F_{\lambda }$. In
fact, for each $\lambda $ take $\lambda _{0}\geq \lambda $ and $\mathcal{U}%
_{0}\in \mathcal{O}$ such that $\left( \lambda _{0},\mathcal{U}_{0}\right)
\in \Gamma $. The subnet $\left( x_{\left( \lambda ^{\prime },\mathcal{U}%
^{\prime }\right) }\right) _{\left( \lambda ^{\prime },\mathcal{U}^{\prime
}\right) \geq \left( \lambda _{0},\mathcal{U}_{0}\right) }$ also converges
to $x$. As $\lambda _{0}\geq \lambda $, $x_{\left( \lambda ^{\prime },%
\mathcal{U}^{\prime }\right) }=x_{\lambda ^{\prime }}\in F_{\lambda ^{\prime
}}\subset F_{\lambda }$, for all $\left( \lambda ^{\prime },\mathcal{U}%
^{\prime }\right) \geq \left( \lambda _{0},\mathcal{U}_{0}\right) $. Since $%
F_{\lambda }$ is closed, it follows that $x\in F_{\lambda }$. Therefore $%
x\in \bigcap\limits_{\lambda }F_{\lambda }$ as desired. Now, as $\gamma
\left( F_{\lambda }\right) \rightarrow \mathcal{O}$, we have $\gamma \left(
\bigcap\limits_{\lambda }F_{\lambda }\right) =\mathcal{O}$. By Proposition %
\ref{P9}, this means that $\bigcap\limits_{\lambda }F_{\lambda }=\mathrm{cls}%
\left( \bigcap\limits_{\lambda }F_{\lambda }\right) $ is compact.

As to the converse, suppose that every decreasing net $\left( F_{\lambda
}\right) $ of nonempty closed sets of $X$, with $\gamma \left( F_{\lambda
}\right) \rightarrow \mathcal{O}$, has nonempty compact intersection. Let $%
\left( x_{\lambda }\right) _{\lambda \in \Lambda }$ be a Cauchy net. For
each $\lambda $, define the set $H_{\lambda }=\left\{ x_{\lambda ^{\prime
}}:\lambda ^{\prime }\geq \lambda \right\} $. Then $H_{\lambda }$ is a
nonempty subset of $X$ and $H_{\lambda _{1}}\subset H_{\lambda _{2}}$
whenever $\lambda _{1}\geq \lambda _{2}$. We claim that $\mathrm{D}\left(
H_{\lambda }\right) \rightarrow \mathcal{O}$. In fact, for a given $\mathcal{%
U}\in \mathcal{O}$ there is $\lambda _{0}$ such that $\mathcal{U}\in \rho
\left( x_{\lambda _{1}},x_{\lambda _{2}}\right) $ whenever $\lambda
_{1},\lambda _{2}\geq \lambda _{0}$. If $x_{\lambda _{1}},x_{\lambda
_{2}}\in H_{\lambda }$, with $\lambda \geq \lambda _{0}$, we have $\lambda
_{1},\lambda _{2}\geq \lambda \geq \lambda _{0}$ and then $\mathcal{U}\in
\rho \left( x_{\lambda _{1}},x_{\lambda _{2}}\right) $. Hence $\mathcal{U}%
\in \mathrm{D}\left( H_{\lambda }\right) $ whenever $\lambda \geq \lambda
_{0}$, and therefore $\mathrm{D}\left( H_{\lambda }\right) \rightarrow 
\mathcal{O}$. Now, by Proposition \ref{P9}, it follows that $\alpha \left( 
\mathrm{cls}\left( H_{\lambda }\right) \right) \rightarrow \mathcal{O}$.
Since $\mathrm{cls}\left( H_{\lambda _{1}}\right) \subset \mathrm{cls}\left(
H_{\lambda _{2}}\right) $ whenever $\lambda _{1}\geq \lambda _{2}$, we
obtain a decreasing net $\left( \mathrm{cls}\left( H_{\lambda }\right)
\right) $ of nonempty closed sets of $X$ with $\alpha \left( \mathrm{cls}%
\left( H_{\lambda }\right) \right) \rightarrow \mathcal{O}$. By hypothesis,
the intersection $\bigcap\limits_{\lambda }\mathrm{cls}\left( H_{\lambda
}\right) $ is nonempty. Then we can take a point $x\in
\bigcap\limits_{\lambda }\mathrm{cls}\left( H_{\lambda }\right) $. We claim
that $x_{\lambda }\rightarrow x$. We claim that $x_{\lambda }\rightarrow x$.
By Proposition \ref{P10}, it means to prove that $\rho \left( x_{\lambda
},x\right) \rightarrow \mathcal{O}$. For a given $\mathcal{U}\in \mathcal{O}$%
, there is $\lambda _{0}$ such that $\lambda \geq \lambda _{0}$ implies $%
\mathcal{U}\in \mathrm{D}\left( \mathrm{cls}\left( H_{\lambda }\right)
\right) $. As $x_{\lambda },x\in \mathrm{cls}\left( H_{\lambda }\right) $,
it follows that $\mathcal{U}\in \rho \left( x_{\lambda },x\right) $ whenever 
$\lambda \geq \lambda _{0}$. Hence $\rho \left( x_{\lambda },x\right)
\rightarrow \mathcal{O}$. Therefore the Cauchy net $\left( x_{\lambda
}\right) $ is convergent and $X$ is a complete admissible space.
\end{proof}

\section{Conclusion and further applications}

Admissible structure in completely regular space allows the extension of the
classical Cantor--Kuratowski intersection theorem (Theorem \ref{T7}). There
is a special interest in applications to topological dynamics. Historically,
the admissible spaces were provided with the intention of extending Conley's
theorems in dynamical systems. Successfully, a topological space endowed
with an admissible family of open coverings became a fundamental structure
for studies of chain recurrence, attraction, and stability (see e.g. \cite%
{BS}, \cite{BSRocha}, \cite{BSRocha2}, \cite{patrao2}, \cite{So2}, \cite{So3}%
). Furthermore, the admissible structure is currently used in studies of
chaos, entropy, and global attractors for semigroup actions on topological
spaces. An important application of the Cantor--Kuratowski theorem will
appear latter (\cite{Richard2}).

\end{document}